\documentclass[a4paper,10pt]{article}
\usepackage{amsfonts,amsmath}
\usepackage[utf8]{inputenc}
\usepackage[T1]{fontenc}
\usepackage{graphicx}
\usepackage{geometry}
\usepackage{float}
\usepackage{amsmath}
\usepackage{amsthm}
\usepackage{amsfonts}
\usepackage{amssymb}
\usepackage{MnSymbol}
\usepackage{mathdots}
\usepackage{stmaryrd}
\usepackage{enumitem}
\usepackage{bm}
\usepackage[]{caption}
\newtheorem{prop}{Proposition}[section]
\newtheorem{lemma}{Lemma}[section]
\newtheorem{thm}{Theorem}[section]
\newtheorem{cor}{Corollary}[section]
\newtheorem{rem}{Remark}[section]
\newtheorem{conj}{Conjecture}[section]
\newtheorem{defi}{Definition}[section]

\newcommand{\w}{\varpi}
\renewcommand{\l}{\lambda}
\newcommand{\Tr}{\mathrm{Tr}}
\renewcommand{\det}{\mathrm{det}}
\newcommand{\diag}{\mathrm{diag}}
\newcommand{\e}{\epsilon}
\newcommand{\C}{\mathbb{C}}
\renewcommand{\d}{\delta}
\newcommand{\D}{\Delta}
\renewcommand{\P}{\mathcal{P}}
\newcommand{\Hom}{\mathrm{Hom}}
\newcommand{\GL}{\mathrm{GL}}
\newcommand{\Gal}{\mathrm{Gal}}

\newcommand{\Ind}{\mathrm{Ind}}
\newcommand{\M}{\mathcal{M}}
\newcommand{\K}{\mathcal{K}}
\renewcommand{\O}{\mathcal{O}}
\title{The split case of the Prasad--Takloo-Bighash conjecture for cuspidal representations of level zero}
\author{M. Chommaux and N. Matringe}

\begin{document}

\maketitle

\begin{abstract}
Let $E/F$ be a quadratic extension of non archimedean local fields of odd residual characteristic. We prove a conjecture 
of Prasad and Takloo-Bighash, in the case of cuspidal representations of depth zero of $\GL_{2m}(F)$. This conjecture characterizes distinction for the pair $(\GL_{2m}(F),\GL_m(E))$ with respect to a character $\mu\circ \det$ of $\GL_m(E)$, in terms of certain conditions on Langlands paremeters, including an epsilon value. We also compute the multiplicity of the involved equivariant linear forms when $E/F$ is unramified, and also when $\mu$ is tame. In both cases this multiplicity is at most one.
\end{abstract}

\section*{Introduction}

Let $E/F$ be a quadratic extension of non archimedean local fields. Let $D$ be an $F$-division algebra of dimension $d^2$ and $n$ be a positive integer such that $nd$ is even. Set $\M=\M(n,D)$, so that $E$ embeds into $M$ uniquely up to inner automorphism. Set 
$C_E(\M)$ to be the centralizer of $E$ in $\M$, it is an $E$-central simple algebra. Let $G=\M^\times$ and $H=C_E(\M)^\times$, for $\mu:E^*\rightarrow \C^*$ a smooth character, we denote by $\mu$ of the
character $H$ 
obtained by composing $\mu$ with the reduced norm on $H$. This paper is concerned with the following conjecture of Prasad and Takloo--Bighash \cite[Conjecture 1]{PTB} 
(the generic transfer assumption in [ibid.] has been shown to be unnecessary in \cite{Suz.20}): 

\begin{conj}\label{conj}\begin{sloppypar}
Let $\pi$ be an irreducible admissible representation of $G=\GL(n,D)$ with central 
character $\omega_\pi$. Let $\mu$ be a character of  
$E^\times$ such that $\mu^{\frac{nd}{2}}|_{F^\times}=\omega_\pi$. If the representation $\pi$ is $\mu$-distinguished by $H$, i.e. if $\Hom_H(\pi,\mu)\neq 0$, then:
\end{sloppypar}
\begin{enumerate}
\item the Langlands parameter $\phi(\pi)$ of $\pi$ takes values in $GSP_{nd}(\mathbb{C})$, with similitude factor $\mu_{|F^\times}$;
\item the epsilon factor satisfies the relation 
\[\epsilon(\frac{1}{2},\phi(\pi)\otimes \Ind_{W_E}^{W_F}(\mu^{-1}))=(-1)^n\omega_{E/F}(-1)^{\frac{nd}{2}}\mu(-1)^{\frac{nd}{2}}\] 
where $\omega_{E/F}$ is the quadratic character of $F^\times$ with kernel the norms of $E^\times$, and $W$ stands for the Weil group.
\end{enumerate}
If $\pi$ is a disrete series representation of $G$, then the implication becomes an equivalence.
\end{conj}

This conjecture is inspired by earlier results of J. Tunnel and H. Saito for $n=2$ and $D=F$. In fact Tunnel was the first to consider the problem for $\GL(2,F)$, and 
to solve it when the residual characteristic of $F$ is not $2$ (\cite[Theorem p.1277]{T.83}), then Saito found a simpler proof valid in characteristic
different from $2$ 
(\cite[Theorem p.99]{S.93}). 

The current status of the conjecture is the following: when $\mu=1$ and $F$ has characteristic zero and odd residual characteristic, the conjecture 
should be proved by the combination of 
\cite{X.20}, \cite{Sec.20}, \cite{Suz.20} and \cite{SX.20}. The paper \cite{X.20} holds in characteristic zero, the paper \cite{Suz.20} 
holds in characteristic not $2$, and the parts of \cite{Sec.20} which do not 
depend on \cite{X.20} hold in residual characteristic not $2$. Finally the reduction from discrete series to cuspidal representations done in 
\cite{SX.20} holds in characteristic zero (and assumes odd residual characteristic because the main theorem depends on \cite{Sec.20}).

For general $\mu$ and $F$ of characteristic not $2$, the conjecture is proved by the first named author in \cite{C.19} for Steinberg representations. In this paper, 
when the residual characteristic of $F$ is not $2$, we prove it for general $\mu$ and depth-zero cuspidal representations of $F$-split $G$.\\

Let us describe the how the paper is organized: we assume the residual characteristic of $F$ to be different from $2$, and suppose that $n\geq 4$ as in any case the conjecture we intend to prove is known for $n=2$ from Tunnel and Saito's results.\\

In Section \ref{section distinction when mu is tame} we treat the case where $\mu$ is tame. By standard Mackey theory arguments, and an also standard argument of Hakim and Murnaghan, we characterize $\mu$-distinction of depth-zero cuspidal representations in terms of their Langlands parameters (Theorem \ref{theorem distinction when mu is tame}).\\ 

In Section \ref{section mu distinguished implies mu selfdual}, in order to characterize distinction when $\mu$ is not tame, we prove in 
Proposition \ref{proposition selfdual cuspidal representations} that a 
$\mu$-distinguished cuspidal representation of any inner form of $\GL_n(F)$ is $\mu$-selfdual, by a standard globalization argument.\\

In Section \ref{section mu distinction} we extend in Theorem \ref{theorem characterization of distinction} our characterization of $\mu$-distinction depth-zero cuspidal representations of $\GL_n(F)$ in terms of their Langlands parameter to any character $\mu$. Along the way we isolate the contribution of residual twisted Shalika models in Proposition \ref{propositon example of non trivial double coset contribution}, and show in Proposition \ref{proposition multiplicity one unramified} that when $E/F$ is unramified, the only double coset contributing to distinction is the one isolated in Proposition \ref{propositon example of non trivial double coset contribution}. In particular this gives a multiplicity at most one statement when $E/F$ is unramified.\\

In Section \ref{section mu simplecticity} we give an explicit characterization of $\mu$-simplecticity of depth-zero cuspidal representations of $\GL_n(F)$ (see 
Corollary \ref{corollary mu-symplecticity}), which 
resembles (and in fact is implied by) our $\mu$-distinction criterion. \\

Finally in Section \ref{section PTB} we prove the Prasad and Takloo-Bighash conjecture for depth-zero cuspidal representations of $\GL_n(F)$ (Corollary \ref{corollary PTB}). With all the analysis done before, it reduces to a pleasant computation of the epsilon value of the conjecture for $\mu$-symplectic depth-zero cuspidal representations of $\GL_n(F)$ (with an extra condition on the central character) which is done in particular thanks to a result of Fröhlich and Queyrut (\cite{FQ.73}). The computation in question is performed in Theorem \ref{theorem PTB epsilon value}.\\

\noindent \textbf{Acknowledgements.} We thank P. Broussous, D. Prasad and V. Sécherre  for very useful conversations and comments.
We also thank P. Broussous for his contribution to some parts of the paper, and D. Prasad and V. Sécherre for notifying mistakes in previous versions
of the paper. The second named author thanks the Abdus Salam School of Mathematical Sciences (Lahore) where parts of this paper were written for 
its hospitality. We finally thank the referee for their accurate comments and suggestions, in particular allowing many simplifications in Section \ref{section mu simplecticity}. 

\section{Preliminary results}

\subsection{Notation / definitions}\label{not_cusp}~

Let $F$ be a non-archimedean local field of residual characteristic not $2$ and $D$ an $F$-central division algebra of dimension $d^2$ over $F$. We fix an algebraic closure which will contain all finite extensions of $F$ under consideration, and similarly for the residual field $k_F$ of $F$. For a finite  extension $\bullet$ of $F$, we 
denote by the $\mathcal{O}_\bullet$, $\mathcal{P}_\bullet$, $\varpi_\bullet$, $k_\bullet$ and $q_\bullet$ the ring of integers, its maximal ideal, a fixed uniformizer, the residual field of $\bullet$. Whenever $\chi:\bullet^*\rightarrow \C^*$ is a (smooth) character, we say that it is tame if $\mu(1+\P_\bullet)=\{1\}$.
Let $E$ be a quadratic extension of $F$ (we write $E=F[\delta]$ for a fixed $\delta$ in $E\setminus F$ such that $\delta^2$ is in $F$ and we set $\Delta=\delta^2$). We let $e(E/F)$ denote the ramification index of $E/F$. When $E/F$ is ramified, we choose $\varpi_E$ and $\varpi_F$ such that $\varpi_F=\varpi_E^2$; when 
$E/F$ is unramified, we choose $\varpi_F=\varpi_E$. 

Throughout the paper we will have \[nd=2m\] for $m$ a natural number. In fact except in Proposition \ref{proposition selfdual cuspidal representations}, we will have 
\[D=F\Leftrightarrow d=1 \Leftrightarrow n=2m.\] We will consider the group \[G=\GL_n(F)\] and its subgroup \[H\simeq \GL_m(E)\] embedded in $G$ as we now explain. Let $(e_1,\dots,e_m)$ be the canonical basis of $E^m$. Then $E^m$ identifies to $F^n$ as $F$-vector space via 
the basis $\mathcal{B}=(\d e_1,\dots,\d e_m, e_1,\dots,e_m)$. Now $H$ embeds in $G$ as the fixed points of 
$G$ under the involution 
\[\begin{matrix}\theta&:&G&\longrightarrow&G\\
   &&g&\mapsto& AgA^{-1}
\end{matrix} \hspace{1cm}\text{ where }\:A=\begin{pmatrix} & I_m \\ \D I_m &  \end{pmatrix}.\]

We denote by $\det_E$ the determinant map on $H$ identified with $\GL_m(E)$, with values in $E^*$. Hence any character 
$\mu$ of $E^*$ defines a character which we still write $\mu$ of $H$, and in fact all characters of $H$ are such.

\subsection{Parametrization of depth-zero cuspidal representations}

We call a depth-zero cuspidal representation of $\GL_n(F)$ an irreducible cuspidal representation of this group with a vector fixed by $I_n+\varpi_F\M_n(\mathcal{O}_F)$. One can parametrize depth-zero cuspidal representations by admissible tame pairs as we now recall (see \cite[Part 5]{BH.11}).

\begin{itemize}[label=\textbullet]
 \item Let $L/F$ be the unramified field extension of degree $n$, of ring of integers $\mathcal{O}_L$. Let $\chi$ be a character 
 of $L^*$ that satisfies:
 \begin{itemize}
 \item $\chi$ is tame, 
 \item $\chi\circ \gamma=\chi\Rightarrow\gamma=id_L$ for all $\gamma$ in $\mathrm{Gal}_F(L)$; we say that $\chi$ is regular.
 \end{itemize}
 Such a pair $(L,\chi)$ is said to be \textbf{tame admissible}.
 \item As $\chi$ is trivial on $1+\mathcal P_L$, $(L,\chi)$ induces a pair $(k_L,\overline{\chi})$ where $\overline\chi$ is a 
 character of $k_L^*$ which satisfies $\overline\chi\circ\overline\gamma=\overline\chi\Rightarrow\overline\gamma=id_{k_L}$ for all $\overline\gamma$ in 
 $\mathrm{Gal}_{k_F}(k_L)$; $\overline\chi$ is said to be \textbf{regular}.~\\
 By Green parametrization, one can associate to $(k_L,\overline\chi)$ an irreducible cuspidal representation 
 $(\overline\pi_{\overline\chi},\mathcal V)$ of $\GL_n(k_F)$ \textit{i.e.} an irreducible representation of $\GL_n(k_F)$ such that 
 for all proper parabolic subgroup $P$ with Levi decomposition $P=MN$, the vector subspace of fixed points of $\mathcal V$ by $N$ 
 is trivial.\\
 More precisely, if one defines an equivalence relation $\sim$ on regular characters of $k_L^*$ by
 $$\bar\chi_1\sim\bar\chi_2\text{ if and only if }\exists \bar\gamma\in\mathrm{Gal}_{k_F}(k_L)\text{ such that }\bar\chi_2=\bar\chi_1\circ\bar\gamma,$$
 one has a bijection:
 \begin{align*}
  \left\{
  \begin{gathered}
    \text{equivalence classes for }\sim\\
    \text{of regular characters of }k_L^*
   \end{gathered}
  \right\}
  &\longrightarrow&
  \left\{
    \begin{gathered}
      \text{equivalence classes of irreducible }\\
      \text{cuspidal representations of }\GL_n(k_F)
  \end{gathered}
  \right\}\\
  \bar\chi\:\:\:&\mapsto&\bar\pi_{\bar\chi}\hspace{5.8cm}
\end{align*}
\item \begin{sloppypar} As $\GL_n(k_F)\simeq \GL_n(\mathcal{O}_F)/1+\varpi_F\mathcal{M}_n(\mathcal{O}_F)$, $\bar\pi_{\bar\chi}$ can be seen as a 
representation of $\GL_n(\mathcal{O}_F)$ that is trivial on $1+\varpi_F\mathcal M_n(\mathcal O_F)$. Then, one can define a 
representation of $F^*\GL_n(\mathcal{O}_F)$, denoted by $\lambda_\chi$, in the following way:\end{sloppypar}
$$\lambda_\chi(xk)=\chi_{|F^*}(x)\bar\pi_{\bar\chi}(k)\text{ for all }x\in F^*, k\in \GL_n(\mathcal{O}_F).$$
\item Finally, we set $\pi(\chi):=c-\Ind_{F^*\GL_n(\mathcal{O}_F)}^{\GL_n(F)}(\lambda_\chi)$ ($c-\Ind$ refers to compact induction), it 
is a depth zero cuspidal representation of $G$. If we denote again by $\sim$ the equivalence relation between admissible tame pairs of 
degree $n$ by
$$\chi_1\sim\chi_2\text{ if and only if }\exists\gamma\in\mathrm{Gal}_F(L)\text{ such that }\chi_2=\chi_1\circ\gamma,$$
one gets a bijection:
\begin{align*}
  \left\{
  \begin{gathered}
    \text{equivalence classes for }\sim\text{ of }\\
    \text{admissible tame pairs of degree }n
   \end{gathered}
  \right\}
  &\longrightarrow&
  \left\{
    \begin{gathered}
      \text{equivalence classes of }\\
      \text{depth $0$ cuspidal }\\
      \text{representations of }\GL_n(F)
  \end{gathered}
  \right\}\\
  (L,\chi)\:\:\:&\mapsto&\pi(\chi)\hspace{4cm}
\end{align*}
\end{itemize}
Let us recall that the central character of $\pi(\chi)$ is $\chi_{|F^*}$ and its contragredient is $\pi(\chi)^\vee\simeq \pi(\chi^{-1})$.

\subsection{Reminder about the building of $\GL_n(F)$}

Let us recall how to describe the Bruhat-Tits building of $\GL_n(F)$ with lattice chains. 

\begin{defi}
 An $\bm{\mathcal O_F}$-lattice chain in $F^n$ is a strictly decreasing sequence (for inclusion) 
  $\mathcal L=(L_k)_{k\in\mathbb Z}$ of lattices such that there exists a unique positive integer $T$ that satisfies: for any uniformizer 
  $\varpi_F$, $\varpi_F L_k=L_{k+T}$ for all $k\in\mathbb Z$. The integer $T$ (or $T(\mathcal L)$) is called the period of 
  $\mathcal L$.
\end{defi}

It is known that $T$ is at most $n$, and that there are lattice chains with period $n$. The group $\GL_n(F)$ naturally acts on the set of lattice chains
$(L_k)_{k\in\mathbb{Z}}$ by $g\cdot(L_k)_{k}=(g\cdot L_k)_k$ for $g\in \GL_n(F)$, and we say that two lattice chains are \textit{equivalent} if they are in the same $Z$-orbit, for $Z$ the center of $\GL_n(F)$. 

\begin{defi}
 As a simplicial complex, the Bruhat-Tits building of $\GL_n(F)$, $X_G$, is defined as the the set of equivalence classes of lattice chains. The $(T-1)$-dimensional simplex being the equivalence classes of lattice chains of period $T$. 
\end{defi}

We identify lattice chains of period one with $Z$-orbits of lattices in $F^n$, and denote by $[L]$ the $Z$-orbit of the lattice $L$: by definition they from the set $X_G^\circ$ of \textit{vertices of $X_G$}. Clearly the group $\GL_n(F)/Z$, hence $\GL_n(F)$ acts on $X_G$ by respecting its simplicial 
structure. Let $\mathcal K$ denote the maximal compact modulo center subgroup $F^*\GL_n(\mathcal{O}_F)$ and let $s_0$ be the vertex of $X_G$ that is stabilized by $\mathcal K$ i.e. the standard lattice chain of period $1$; the vertex $s_0$ is called the \textit{standard vertex} of $X_G$. We recall the following $G$-set isomorphism: 
\begin{equation}\label{iso_immeuble}
\begin{matrix}X_G^\circ&\overset{\sim}{\longrightarrow}&G/\mathcal{K}\\
g\cdot s_0&\longmapsto&g\mathcal{K}   
  \end{matrix}\hspace{1cm}\text{for }g\in G.
  \end{equation}

We will need the geometric realization of $X_G$, denoted by $|X_G|$. Each $T-1$-dimensional simplex of $X_G$ is embedded in 
$\mathbb R^{T-1}$ with the following property: if we consider a $T-1$-dimensional simplex, the points of its geometric realization in $|X_G|$ are given by the set of all barycenters of its vertices. We will use the geometric realization of the building 
$X_G$ given by lattice-functions. The definition comes from Section I.2 of \cite{BL.02}.

\begin{defi}
 A \textbf{lattice-function} of $F^n$ is a map ${\Lambda\::\:\mathbb R\longrightarrow\{\text{lattices of }F^n\}}$ satisfying:
 \begin{itemize}[label=\textbullet]
  \item $\varpi_F\Lambda(r)=\Lambda(r+1)$;
  \item $\Lambda$ is decreasing: for all $r\geq s$, $\Lambda(r)\subseteq\Lambda(s)$;
  \item $\Lambda$ is left-continuous for the discrete topology on lattices.
 \end{itemize}
\end{defi}

Let us explain with more details how the set of lattice-functions allows to realize geometrically the building 
of $\GL_n(F)$. Let $\Lambda$ be a lattice-function of $F^n$, then its image is a lattice chain $\mathcal L=(L_k)_{k\in\mathbb Z}$ 
with period $T$. If we denote by $\lambda_k$ the length of the interval defined by $\{r\in\mathbb R,\Lambda(r)=L_k\}$, then the 
point $x_\Lambda$ of $|X_G|$ associated to $\Lambda$ is the barycenter of the weighted points
$([L_0],\lambda_0),([L_1],\lambda_1),\dots,([L_{T-1}],\lambda_{T-1})$.
Two lattice-functions $\Lambda_1$ and $\Lambda_2$ are said to be \textit{equivalent} if there exists a real number $r_0$ such that 
$\Lambda_1(r)=\Lambda_2(r+r_0)$ for all $r\in\mathbb R$, in which case they realize the same point of the building. We denote by 
$\overline\Lambda$ the class of a lattice-function $\Lambda$. 
Moreover, the group $\GL_n(F)$ naturally acts on the set of lattice-functions by: $(g\cdot\Lambda)(r)=g\cdot(\Lambda(r))$ for 
every lattice-function $\Lambda$, every $g\in \GL_n(F)$ and every real number $r$. Thus, one has the following $G$-set isomorphism:
\[\begin{matrix}\{\text{equivalence classes of lattice-functions of $F^n$}\}&\overset{\sim}{\longrightarrow}&|X_G|\\
   \overline\Lambda&\longmapsto&x_\Lambda
  \end{matrix}\]

Of course, all these reminders are valid for the construction of the building of $\GL_m(E)$, $X_H$.

\subsection{Vertices of the building fixed by the involution}

First we recall the relation between $|X_G|$ and $|X_H|$, we will use the following terminology from type theory.

\begin{defi}
 Let $u\in G$ such that $F_1:=F[u]$ is a field; let us denote by $v_{F_1}$ the normalized valuation of $F_1$ and by $e(F_1/F)$ 
 the ramification index of $F_1/F$. One says that $u$ is minimal on $F$ if: 
 \begin{enumerate}
  \item $\mathrm{gcd}(v_{F_1}(u),e(F_1/F))=1$,
  \item $\varpi_F^{-v_{F_1}(u)}u^{e(F_1/F)}+\mathcal{P}_{F_1}$ generates the residual field extension $k_{F_1}/k_F$.
 \end{enumerate}

\end{defi}

Recall that $E=F[\delta]$ and let us show that $\delta$ can be chosen minimal. 

\begin{itemize}[label=\textbullet]
 \item If $E/F$ is ramified, we recall that $\varpi_F:= \varpi_E^2$. 
 If we choose $\delta=\varpi_E$, then we do have $E=F[\varpi_E]$ and $\delta$ is minimal. Indeed, $v_E(\varpi_E)=1$ so 
 $\mathrm{gcd}(v_E(\varpi_E),e(E/F))=1$ and moreover $\varpi_F^{-1}\varpi_E^{2}+\mathcal{P}_{E}=1+\mathcal{P}_E$ which generates 
 $k_E/k_F$ (because $k_E=k_F$ in the ramified case).
 \item If $E/F$ is unramified (\textit{i.e.} $e(E/F)=1$), then $k_E$ is an extension of $k_F$ with cardinality $q_F^2$ 
 and there exists $\xi\in E^*$ a primitive $(q_F^2-1)^{\text{th}}$ root of unity which generates $E$ over $F$. Set
 $\delta:=\xi^\frac{q_F+1}{2}$. As the order of $\delta$ is $2(q_F-1)$, then $\delta\notin F$ but $\delta^2\in F$, so that we do have 
 $E=F[\delta]$ with $\delta^2\in F$. Moreover, $\delta$ is a minimal element because $v_E(\delta)=0$ so $\mathrm{gcd}(v_E(\delta),e(E/F))=1$ 
 and moreover, $\varpi_F^{0}\delta^{1}+\mathcal{P}_{E}=\delta+\mathcal{P}_E$ generates $k_E/k_F$ (see Theorem 7 and Corollary 3 of 
 Chapter 1, \S 4 of Weil \cite{W.74}).
\end{itemize}

From now on, we choose $\delta=\varpi_E$ if $E/F$ is ramified and $\delta=\xi^\frac{q_F+1}{2}$ (for $\xi$ a primitive 
$(q_F^2-1)^{\text{th}}$ root of unity) if $E/F$ is unramified, thus $\delta$ is minimal. Then by \cite[Lemma XII.4.2]{BS.17} we have:

\begin{lemma}\label{B.S}
We have $|X_G|^\theta=|X_G|^{E^*}$.
\end{lemma}

Note that an $\mathcal O_E$-lattice of $E^m$ can always be seen as an  $\mathcal O_F$-lattice of $F^{2m}$ because $\mathcal O_E$ is an $\mathcal O_F$-lattice in $F^2$. Theorem 1.1 of \cite{BL.02} then asserts:

\begin{thm}\label{B.L}
 \begin{enumerate}
  \item There exists a unique map $j\::\: |X_H|\longrightarrow |X_G|$ that is $H$-equivariant and affine.
  \item It is injective and $j(|X_H|)=|X_G|^{E^*}$, the set of points that are fixed by $E^*$.
  \item\label{item 3} If $x\in |X_H|$ is associated to the lattice-function $r\mapsto\Lambda(r)$, then $j(x)$ is associated to the lattice-function
  $r\mapsto\Lambda(e(E/F)r)$.
 \end{enumerate}
\end{thm}

The theorem above enables us to determine the $H$-orbits of $\theta$-fixed vertices in ${X_G}^\circ$ depending on the ramification of $E/F$. 

\begin{prop}\label{proposition stable vertices}
When $E/F$ is unramified, the set $(X_G^\circ)^\theta$ consists of a unique $H$-orbit, namely that of the standard vertex $s_0$ fixed by $\K$, whereas 
when $E/F$ is ramified $(X_G^\circ)^\theta$ is empty.
\end{prop}
\begin{proof}
When $E/F$ is unramified, the map $j$ is simply the identity on lattice-functions and is simplicial. Thus by , ${(X_G^\circ)}^\theta=j(X_H^\circ)$ whence $H\backslash {(X_G^\circ)}^\theta=H\backslash j(X_H^\circ)=j(H\backslash X_H^\circ)$ by Theorem \ref{B.L}. As $H$ acts transitively on $X_H^\circ$, we deduce that $(X_G^\circ)^\theta$ consists of a unique $H$-orbit. Moreover it is that of $s_0$ because $s_0$ is the image of the standard vertex in $X_H^\circ$ under $j$. When $E/F$ is ramified, then by Theorem \ref{B.L} the map $j$ sends an equivalence class of lattice functions with image a lattice chain of of period $1$ to an equivalence class of lattice functions with image a lattice chain of of period $e(E/F)=2$, i.e. it sends a vertex to an interior point of a simplex of dimension $\geq 1$, so $j(X_H)\cap X_G^\circ$ is empty and the resutl follows again from Theorem \ref{B.L}.
\end{proof}

\subsection{Properties of local constants}\label{section properties of constants}

Let $K'/K$ be a finite separable extension of non-archimedean local fields, if $\psi$ is a non-trivial character of $K$, we denote by $\psi_{K'}$ the character $\psi\circ \Tr_{K'/K}$. We call \textit{the conductor of $\psi$} the smallest integer $d(\psi)$ such that $\psi$ is trivial on $\P_K^{d(\psi)}$. Similarly if $\chi$ is a character of $K^*$, we call 
\textit{the conductor of $\chi$} the integer $c(\chi)$ equal to zero if $\chi$ is unramified, or equal to the smallest 
integer such that $\chi$ is trivial on $1+\P_{K'}^{c(\psi)}$ if $\chi$ is ramified. We say that $\chi$ is tame when 
$c(\chi)\leq 1$. When 
\textit{${K'}/K$ is unramified}, it follows from \cite[Chapter 8, Corollary 3]{W.74} that 
\begin{equation}\label{equation equal conductors}
d(\psi_{K'})=d(\psi).
\end{equation}
 If $\phi$ is a representation of $W_K$ of finite dimension, and $\psi$ is a non-trivial character of $K$, we refer to \cite[3.6.4]{T.79} for the definition of the root number 
$\e(1/2,\phi,\psi)$ (denoted $\e_L$ there). 
One then defines the Langlands $\l$-constant: \[\l({K'}/K,\psi)=\frac{\e(1/2,\Ind_{W_{K'}}^{W_K}(\mathbf{1}_{W_K}),\psi)}{\e(1/2,\mathbf{1}_{W_L},\psi_{K'})}.\] We set \[\omega_{{K'}/K}=\det\circ\Ind_{W_{K'}}^{W_K}(\mathbf{1}_{W_K}),\] it identifies with the quadratic character of $K^*$ with kernel the norms of ${K'}^*$ when ${K'}/K$ is quadratic. For $a\in K^\times$, we set $\psi_a=\psi(a\ . \ )$. These constants enjoy the following list of properties, which we will freely use later in the paper.

\begin{enumerate}
\item \label{equation additivity of epsilon} $\e(1/2,\phi\oplus \phi',\psi)=\e(1/2,\phi,\psi)\e(1/2,\phi',\psi)$ where $\phi'$ is another finite dimensional representation of $W_K$ \cite[(3.4.2)]{T.79}.
\item \label{equation translation caractère additif} $\e(1/2,\phi,\psi_a)=\det(\phi(a))\e(1/2,\phi,\psi)$ (\cite[(3.6.6)]{T.79}).
\item \label{equation galois invariance of epsilon} $\e(1/2,\phi^{\sigma},\psi^\sigma)=\e(1/2,\phi,\psi)$ whenever $\sigma$ is a finite order field automorphism of $K$, as can be checked by the definition of the epsilon factor.
\item \label{equation epsilon times epsilon dual} $\e(1/2,\phi,\psi)\e(1/2,\phi^\vee,\psi^{-1})=1$ (\cite[(3.6.7)]{T.79}).
\item \label{equation torsion NR} If $\chi$ is a character of $K^*$, and $\mu$ is an unramified character of $K^*$, by  \cite[(3.6.5)]{T.79}: 
\[\e(1/2,\mu\chi,\psi)=\mu(\w_K^{d(\psi)+c(\chi)})\e(1/2,\chi,\psi).\]
\item \label{equation FQ} If ${K'}/K$ is a quadratic, $\d\in \ker(\Tr_{{K'}/K})-\{0\}$, $\chi$ is a character of ${K'}^*$ with $\chi_{|K^*}=1$, then by \cite[Theorem 3]{FQ.73}: \[\e(1/2,\chi,\psi_{K'})=\chi(\d).\]
\item \label{equation inductivity} If $\phi_{K'}$ is an $r$-dimensional representation of $W_{K'}$, then 
\[\e(1/2,\Ind_{W_{K'}}^{W_K}(\phi_{K'}),\psi)=\l({K'}/K,\psi)^r\e(1/2,\phi_{K'},\psi_{K'})\] 
(\cite[(30.4.2)]{BH.06}).
\item \label{equation constante de Langlands NR} If ${K'}/K$ is unramified with $[K'/K]=n$: \[\l({K'}/K,\psi)=(-1)^{d(\psi)(n-1)}\] (for example \cite{M.86} and \ref{equation translation caractère additif}., together with Equation (\ref{equation equal conductors}).)
\item \label{equation multiplicativité de la constante de Langlands} If $K''$ is a field with $K\subset K'' \subset {K'}$, then 
\[\l({K'}/K,\psi)=\l({K'}/K'',\psi_{K''})\l(K''/K,\psi)^{[{K'}:K'']}\] (\cite{L.70}).
\item \label{equation square=-1} $\l({K'}/K,\psi)^2=\omega_{{K'}/K}(-1)$ (\cite[(30.4.3)]{BH.06}).
\end{enumerate}

\section{Distinction of depth-zero cuspidal representations when $\mu$ is tame}\label{section distinction when mu is tame}

This case is the easiest case, and we use the proof of \cite[Proposition 5.20]{HM.08} to determine multiplicities. We fix $\pi(\chi)$ a cuspidal representation of $\GL_n(F)$ of depth-zero, and $\mu$ is a character of $E^*$.

\begin{lemma}[\cite{HM.08}]
 Let $x\in X_G^\circ$ a vertex such that $\theta(x)\neq x$. Let $\mathcal K_x$ be the stabilizer of $x$ in $G$, $K_x$ the maximal compact sugbroup of $\mathcal K_x$ and $K_x^1\subseteq K_x$ its pro-unipotent radical. Let  $\overline\sigma$ be a cuspidal representation of 
 $K_x/K_x^1$, let $\sigma$ be the inflation of $\overline\sigma$ to $K_x$. Suppose that $\mu$ is tame and set $\rho:=\mu$, then $\mathrm{Hom}_{K_x\cap H}(\sigma,\rho)=\{0\}$.
\end{lemma}

\begin{proof}
By the proof of Proposition 5.20 of \cite{HM.08}, if $\theta(x)\neq x$ there is a group $K_x^1\subset U\subset K_x$, such that $\overline U:=U/K_x^1\subset K_x/K_x^1$ is the unipotent radical of a proper parabolic subgroup of 
$K_x/K_x^1\simeq \GL_n(k_F)$ and which satisfies $U=U^\theta K_x^1$ (where the exponent denotes fixed points). Suppose for the sake of contradiction that $\mathrm{Hom}_{K_x\cap H}(\sigma,\rho)\neq\{0\}$, this first implies that ${\rho_{|K_x^1\cap H}=1}$ because $\sigma$ is trivial on $K_x^1$.
Now for $h\in U\cap H$, there exists $\alpha\geq0$ such that $h^{p^\alpha}\in K_x^1\cap H$, which implies that $\rho(h^{p^\alpha})=1$. Thus, $\mu(\det_E(h))^{p^\alpha}=1$ where $\det_E(h)\in\mathcal O_E^\times$. Yet $\mu$ is tame so $\mu_{|\mathcal O_E^\times}$ factors through $\mathcal O_E^\times/(1+\mathcal P_E)$ which is a finite group of order prime to $p$, hence $\mu(det(h))=1$. So $\rho_{|U\cap H}=1$ and \[\{0\}\neq \mathrm{Hom}_{K_x\cap H}(\sigma,\rho)\subset \mathrm{Hom}_{U^\theta}(\sigma,1)\simeq \mathrm{Hom}_{\overline{U}}(\overline{\sigma},1)\] as $U=U^\theta K_x^1$, contradicting the cuspidality of $\sigma$.
\end{proof}

In other words, as each vertex $x$ in $X_G^\circ$ is of the form $g\cdot s_0$ for a certain $g$ in $G$ and its stabilizer is $g\mathcal Kg^{-1}$, this amounts to the following lemma.

\begin{lemma}[\cite{HM.08}]\label{HM}
 If $g\in H\backslash G/\mathcal{K}$ satisfies $\mathrm{Hom}_{H\cap g\mathcal{K}g^{-1}}({}^g\lambda_\chi,\mu)\neq\{0\}$ (where 
 ${}^g\lambda_\chi(x)=\lambda_\chi(g^{-1}xg)$ for all $x$ in $g\mathcal{K}g^{-1}$), then 
 $g\mathcal{K}g^{-1}$ is stable by $\theta$.
\end{lemma}

The next step is:

\begin{lemma}\label{lemma HM} There is an isomorphism of $\mathbb C$-vector spaces:
 \[\mathrm{Hom}_H(\pi(\chi),\mu)\simeq\underset{g\cdot s_0\in H\backslash (X_G^\circ)^\theta}{\prod}\mathrm{Hom}_{H\cap g\mathcal{K}g^{-1}}({}^g\lambda_\chi,\mu).\]
\end{lemma}

\begin{proof}Write successively:
 \begin{eqnarray*}
  \mathrm{Hom}_H(\pi(\chi),\mu)&=&\mathrm{Hom}_H(c-\Ind_{\mathcal{K}}^G(\lambda_\chi),\mu)\\
  &\simeq&\mathrm{Hom}_H(\underset{g\in H/G\backslash \mathcal{K}}{\oplus}c-\Ind_{H\cap g\mathcal{K}g^{-1}}^\mathcal{K}({}^g\lambda_\chi),\mu)\\
  &&\text{by Mackey's restriction formula}\\
  &\simeq&\underset{g\in H\backslash G/\mathcal{K}}{\prod}\mathrm{Hom}_{H\cap g\mathcal{K}g^{-1}}({}^g\lambda_\chi,\mu)\\
  &&\text{by Frobenius reciprocity on the left, for compact induction}\\
  &&\text{from a compact modulo center open subgroup}\\
  &\simeq&\underset{g\cdot s_0\in H\backslash X_G^\circ}{\prod}\mathrm{Hom}_{H\cap g\mathcal{K}g^{-1}}({}^g\lambda_\chi,\mu)\text{ thanks to Isomorphism \eqref{iso_immeuble}}\\
  &\simeq&\underset{g\cdot s_0\in H\backslash (X_G^\circ)^\theta}{\prod}\mathrm{Hom}_{H\cap g\mathcal{K}g^{-1}}({}^g\lambda_\chi,\mu)\text{ thanks to Lemma \ref{HM}}
  \end{eqnarray*}
\end{proof}

We denote by $L_0$ the unramified extension of $F$ of degree $m$. Thanks to Theorem \ref{B.L} and the recent paper \cite{P.19} we obtain:

\begin{thm}\label{theorem distinction when mu is tame}
When $\mu$ is tame and $n\geq 4$, we have $\mathrm{Hom}_H(\pi(\chi),\mu)\neq \{0\}$ if and only if $E/F$ is unramified and $\chi_{|L_0^*}=\mu_{|F^*}\circ N_{L_0/F}$, in which case $\mathrm{Hom}_H(\pi(\chi),\mu)\simeq \C$.
\end{thm}
\begin{proof}
Multiplicity zero in the ramified case is immediate from Lemma \ref{lemma HM} and Theorem \ref{B.L}. When $E/F$ is unramified \ref{lemma HM} and Theorem \ref{B.L} imply that 
\[\mathrm{Hom}_H(\pi,\mu)=\mathrm{Hom}_{H\cap\mathcal{K}}(\lambda_\chi,\mu),\] which is zero 
if $\chi_{|F^*}\neq \mu_{|F^*}^m$. If $\chi_{|F^*}=\mu_{|F^*}^m$ (which is in particular true when $\chi_{|L_0^*}=\mu_{|F^*}\circ N_{L_0/F}$) we obtain \[\mathrm{Hom}_{H\cap\mathcal{K}}(\lambda_\chi,\mu)= \mathrm{Hom}_{H\cap K}(\lambda_\chi,\mu)= 
\Hom_{\overline{H}}(\overline{\pi}_{\overline{\chi}},\overline{\mu}).\] The result then follows from 
\cite[Proposition 4.3]{P.19} (which has the assumption $n\geq 4$).
\end{proof}

\section{On $\mu$-selfduality of $\mu$-distinguished representations}\label{section mu distinguished implies mu selfdual}

Now we take $\mu$ any character of $E^*$ with no restriction on its conductor. We intend to prove that $\mu$-distinguished representations of cuspidal (of any level) representations of any inner form of $\GL_n(F)$ is $\mu$-selfdual automatically. Our result will follow from a classical globalization argument, and the case of 
principal series for split inner forms.

\begin{prop}\label{proposition selfdual principal series}
Let $\pi$ be a generic principal series of $\GL_n(F)$ (induced from a character of a Borel subgroup), and $\mu_1$ be a character of $F^*\times F^*$, and $\mu_2$ be a character of $E^*$. Let $H_1$ be the block diagonal subgroup $\GL_m(F)\times \GL_m(F)$ and $H_2$ be the subgroup $H\simeq \GL_m(E)$ of $\GL_n(F)$. Then if $\pi$ is $\mu_i$-distinguished by $H_i$, then 
\[\pi\simeq {\mu_i}_{|F^*}\otimes \pi^\vee\] (where $F^*$ is diagonally embedded in $F^*\times F^*$ in the first case)
\end{prop}
\begin{proof}
We only do the case $(H_1,\mu_1)$, as the argument for $(H_2,\mu_2)$ is completely similar but simpler due to simplification of quotients of modulus characters (see \cite{C.19} for the parametrization of double cossets involved there, and \cite[(5.3) and Remark 5.4]{BM.19} for the modulus characters involved). 
Here we rather consider distinction by the conjugate $H$ of $H_1$ by the matrix $w_{n}$ of \cite[p.121]{M.15}, and set 
$h(g_1,g_1)=w_{n}^{-1}\diag(g_1,g_2)w_{n}$ for $g_i\in \GL_m(F)$. The character $\mu_1$ is of the form 
$\mu_{\alpha,\beta}(h(g_1,g_1))=\alpha(\det(g_1))\beta(\det(g_2))$ for $\alpha$ and $\beta$ characters of $F^*$. Let $B$ be the 
upper triangular Borel subgroup of $G=\GL_n(F)$ and $\chi$ be a character of the diagonal torus $A$ of $G$ such that 
$\pi=Ind_B^G(\chi)$ is generic. We want to show that if $\pi$ is $\mu_{\alpha,\beta}$-distinguished, then 
\[\pi\simeq \alpha\beta \otimes \pi^\vee.\] This amounts to prove that there is a permutation 
$\sigma\in S_n$ such that \begin{equation}\label{equation sduality} \alpha\beta\chi^{-\sigma}=\chi, \end{equation} where by abuse of notation 
\[(\alpha\beta)(\diag(a_1,\dots,a_n))=\prod_{i=1}^n (\alpha\beta)(a_i).\] 

We will do this by using Mackey theory, i.e. the natural filtration of $Ind_B^G(\chi)_{|H}$ with sub-quotients 
$ind_{u^{-1}Bu\cap H}^{H}((\d_B^{1/2} \chi)^{u^{-1}})$ when $u$ varies through a set of representatives of $B\backslash G/H$ 
(and $(\d_B^{1/2} \chi)^{u^{-1}}:=\d_B^{1/2} \chi(u\ . \ u^{-1})$)
, so that if 
\[\Hom_{H}(Ind_B^G(\chi),\mu_{\alpha,\beta})\neq \{0\}\] then 
some space \[\Hom_H(ind_{u^{-1}Bu\cap H}^{H}((\d_B^{1/2} \chi)^{u^{-1}}),\mu_{\alpha,\beta})\simeq 
\Hom_{B\cap uHu^{-1}}(\d_B^{1/2}\chi,\d_{B\cap uHu^{-1}}\mu_{\alpha,\beta}^u)\subseteq
\Hom_{A\cap uHu^{-1}}(\chi,\frac{\d_{B\cap uHu^{-1}}}{\d_B^{1/2}}\mu_{\alpha,\beta}^u)\] must be nonzero 
(in fact with the representatives $u$ given in \cite[Section 3.2]{M.15} the last inclusion is an equality). 
This will tell us that $\chi$ has to be of a particular form which will give the existence of $\sigma$ such that 
Equation (\ref{equation sduality}) is satisfied.  

 We set \[\e=\diag(1,-1,\dots,1,-1)\in G\] so that $H$ is the subgroup of $G$ fixed under the  conjugation $\theta_{\e}$ by $\e$. 
By a re-interpretation of the discussion in \cite[Section 3.2]{M.15}, the double cosets $B\backslash G/H$ are parametrized by couples $s=(w_s,x_s)$ where $w_s\in S_n\subset G$ is an involution, and $x_s$ is a map from the set of fixed points $\mathrm{Fix}(w_s)$ of $w_s$ in 
$\{1,\dots,n\}$ to $\{\pm 1\}$, such that $|x_s^{-1}(\{-1\})|=|x_s^{-1}(\{1\})|=\frac{|\mathrm{Fix}(w_s)|}{2}$. The corresponding representative $u_s$ in $B\backslash G/H$ in particular satisfies $u_s\e u_s^{-1} \e^{-1}=w_s$, and we set \[\theta_s(x)=w_s\theta_{\e}(x)w_s^{-1}= u_s\theta_{\e}(u_s^{-1} x u_s ) u_s^{-1}\] for $x\in G$. Conjugation by $u_s$ stabilizes $A$, and $\theta_s$ as well. 
Suppose that $\pi$ is $\mu_{\alpha,\beta}$-distinguished, by the Mackey strategy discussed above (see also before Theorem \cite[Theorem 3.14]{M.15}), there is $s$ such that 
\[\chi_{|A^{\theta_s}}=(\d_{B^{\theta_s}}\d_B^{-1/2}\mu_{\alpha,\beta}^{u_s})_{|A^{\theta_s}},\] where the exponent $\theta_s$ denotes the fixed points of $\theta_s$ in the corresponding set (which is not necessarily $\theta_s$-stable, for example $B$), and 
$\mu_{\alpha,\beta}^{u_s}(a_s)=\mu_{\alpha,\beta}(u_s^{-1} a_s u_s)$ for $a_s\in A^{\theta_s}$.
The character $\d_{B^{\theta_s}}\d_B^{-1/2}$ restricted to $A^{\theta_s}$ is computed in \cite[Proposition 3.6]{M.15}. We extend 
$x_s$ from $\mathrm{Fix}(w_s)$ to $\{1,\dots,n\}$ by $0$ outside $\mathrm{Fix}(w_s)$. Then for $a=\diag(a_1,\dots,a_n)\in A^{\theta_s}$ one has: 
\[\d_{B^{\theta_s}}\d_B^{-1/2}(a)=\prod_{1\leq i<j \leq n} |a_i|^{\frac{x_s(i)x_s(j)}{2}}|a_j|^{-\frac{x_s(i)x_s(j)}{2}}.\]
On the other hand by a computation similar to that done in the proof of \cite[proposition 3.6]{M.15}, we have for $a=\diag(a_1,\dots,a_n)\in A^{\theta_s}$ (note that for any $i$ one has $a_{w_s(i)}=a_i$):
\[\mu_{\alpha,\beta}^{u_s}(a)=\prod_{i\in x_s^{-1}(\{1\})}\alpha(a_i)\prod_{i\in x_s^{-1}(\{-1\})}\beta(a_i)
\prod_{i\in x_s^{-1}(\{0\}), \ i<w_s(i)}\alpha\beta(a_i).\]
For $a\in A$ we set $w_s(a)=w_s a w_s^{-1}$, so that $aw_s(a)\in A^{\theta_s}$, then from the relations above it follows that 
for $a\in A$ (note that $x_s\circ w_s=-x_s$ and is order reversing on $\{1,\dots,n\}-\mathrm{Fix}(w_s)$): \[\chi(aw_s(a))=\alpha(a)\beta(a),\] i.e. 
\[\chi\chi^{w_s}=\alpha\beta\] so we can choose the sought $\sigma\in S_n$ to be $w_s$.
\end{proof}

As in \cite[Proposition 5.2]{BM.19}, we deduce from Proposition \ref{proposition selfdual principal series}, using the globalization results of \cite{PSP.08} and \cite{GL.18} together with the strong multiplicity one theorems from \cite{B.08} and \cite{BR.17}, the following result.

\begin{prop}\label{proposition selfdual cuspidal representations}
Let $D$ be an $F$-division algebra of index $d$ and $n$ a positive integer such that $nd$ is even, let $H$ be the centralizer of $E$ in $G=\GL_n(D)$. Let $\mu$ be a character of $E^*$ identified via the reduced norm to a character of $H$, then a cuspidal representation $\pi$ of $G$ which is $\mu$-distinguished satisfies 
\[\pi\simeq \mu_{|F^*}\otimes \pi^\vee.\]
\end{prop}

Here are two important corollaries for depth-zero cuspidal representations of $\GL_n(F)$. 

\begin{cor}\label{corollary mu's restriction is tame}
Let $\pi$ be a cuspidal representation of $\GL_n(F)$ which is of depth zero, and $\mu$-distinguished, then automatically 
$\mu_{|F^*}$ is tame (i.e. $\mu(1+\P_F)=1$). 
\end{cor}
\begin{proof}
Write $\pi=\pi(\chi)$. By Proposition \ref{proposition selfdual cuspidal representations} we have $\chi^\gamma=\mu_{|F^*} \circ N_{L/F}.\chi^{-1}$ for some $\gamma\in \Gal_F(L)$. But because $\chi^\gamma$ and $\chi^{-1}$ are both tame, the result follows from the fact that $N_{L/F}(1+\P_L)=1+\P_F$.
\end{proof}

We denote by $L_0$ the unramified extension of $F$ of degree $m$.

\begin{cor}\label{corollary distinguished determines restriction to L0}
Suppose that $n\geq 4$. Let $\pi(\chi)$ be a cuspidal $\mu$-distinguished representation of $\GL_n(F)$ of depth zero. Then 
\[\chi_{|L_0^*}=\mu_{|F^*} \circ N_{L_0/F}.\]
\end{cor}
\begin{proof}
Thanks to Proposition \ref{proposition selfdual cuspidal representations}, there is $\gamma\in \Gal_F(L)$ such that 
$\chi^\gamma=\mu_{|F^*} \circ N_{L/F} \chi^{-1}$. Because $\chi$ and and $\mu_{|F^*}$ are tame this this reduces to 
$\overline{\chi}^{\overline{\gamma}}=\overline{\mu_{|F^*}}\circ N_{L/F} \overline{\chi}^{-1}$. This implies that $\chi^{\gamma^2}=\chi$, hence that $\gamma$ has order dividing two because $\chi$ is regular. If $\gamma$ was trivial one would have $\chi^2=\mu_{|F^*} \circ N_{L/F}$. Because $\chi$ and $\mu_{|F^*}$ are tame this would imply 
\[\overline{\chi}^2=\overline{\mu}\circ N_{k_L/k_F}.\] But the group of characters of the form $\alpha\circ  N_{k_L/k_F}$ for $\alpha$ a character of $k_F^*$ form a group of order $q_F-1$ so one should have $\overline{\chi}^{2(q_F-1)}=1$. But because 
$\chi$ is regular and $n\geq 4$, the character $\overline{\chi}^{q_F^2-1}$ must be nontrivial, hence $\overline{\chi}^{2(q_F-1)}\neq 1$. Thus $\gamma$ is the conjugation of $L/L_0$ so $\chi\circ N_{L/L_0}=\mu_{|F^*} \circ N_{L/F}$, and 
$\chi$ and $\mu_{|F^*} \circ N_{L_0/F}$ agree on the units of $L_0^*$ because $L/L_0$ is unramified. Finally they also agree on $\w_F$ by central character considerations.
\end{proof}

\section{Distinction of depth-zero cuspidal representations}\label{section mu distinction}
 
We want to show that the necessary condition obtained in the above section is also sufficient when $\mu$ is not tame. By Proposition \ref{proposition stable vertices} and Lemma \ref{HM}, 
the contribution to distinction in Mackey formula will in this case arise from double cosets in $H\backslash G /\K$ corresponding to $H$-orbits of non $\theta$-fixed vertices of $X_G$. For such double cosets, the distinction problem reduces residually to the existence of a twisted Shalika model, which have been studied by Prasad in \cite{P.00}. We recall his result.

\subsection{Twisted Shalika models over finite fields}

Let $\pi$ be an irreducible representation of $\GL_n(k_F)$, and $\alpha$ be a character of $k_F^*$, and $\psi$ be a nontrivial character of $k_F$. We recall that we call the Shalika subgroup of $\GL_n(k_F)$ the group: 
\[S_n(k_F)=\{\begin{pmatrix} g & \\ & g \end{pmatrix}\begin{pmatrix} I_m & x \\ & I_m \end{pmatrix}, g\in \GL_m(k_F),\ x\in \M_m(k_F)\} .\] On then defines the character $\Psi_\alpha$ of $S_n(k_F)$ by the formula:
\[\Psi_\alpha(\begin{pmatrix} g & \\ & g \end{pmatrix}\begin{pmatrix} I_m & x \\ & I_m \end{pmatrix})=\alpha(\det(g))\psi(\Tr(x)).\]
We say that $\pi$ has an \textit{$\alpha$-twisted Shalika model} if 
\[\Hom_{S_n(k_F)}(\pi,\Psi_\alpha)\neq 0,\] and this does not depend on the choice of $\psi$. The following proposition is due to Prasad (\cite[Theorem 1]{P.00}).

\begin{prop}\label{proposition characterization of Shalika model in terms of Green}
Let $\overline{\pi}_{\overline{\chi}}$ be a cuspidal representation of $\GL_n(k_F)$, then $\overline{\pi}_{\overline{\chi}}$ has an $\alpha$-twisted Shalika model if and only if $\overline{\chi}_{|k_{L_0}^*}=\alpha\circ N_{L_0/F}$ in which case $\Hom_{S_n(k_F)}(\overline{\pi}_{\overline{\chi}},\Psi_\alpha)\simeq \C$.
\end{prop}
\begin{proof}
We denote by $N$ the subgroup of matrices $n(x)=\begin{pmatrix} I_m & x \\ & I_m \end{pmatrix}$ in $\GL_n(k_F)$ and by $(\overline{\pi}_{\overline{\chi}})_{N,\Psi}$ the quotient of $\overline{\pi}_{\overline{\chi}}$ by 
$\{v-\psi(\Tr(x))v, \ n(x) \in N, \ v\in \overline{\pi}_{\overline{\chi}}\}$.  The space $(\overline{\pi}_{\overline{\chi}})_{N,\Psi}$ is a $\GL_n(k_F)$-module (for diagonal action). Then by \cite[Theorem 1]{P.00}, we have 
\[(\overline{\pi}_{\overline{\chi}})_{N,\Psi}=\Ind_{k_{L_0}^*}^{\GL_n(k_F)}(\overline{\chi}_{|k_{L_0}^*}).\] Now by definition 
we have \[\Hom_{S_n(k_F)}(\overline{\pi}_{\overline{\chi}},\Psi_\alpha)\simeq \Hom_{\GL(m,k_F)}(\Ind_{k_{L_0}^*}^{\GL_n(k_F)}(\overline{\chi}_{|k_{L_0}^*}), \alpha\circ \det)\] and this latter space is isomorphic to 
\[\Hom_{k_{L_0}^*}(\overline{\chi}_{|k_{L_0}^*}, \alpha\circ \det)=\Hom_{k_{L_0}^*}(\overline{\chi}_{|k_{L_0}^*}, \alpha\circ N_{L_0/F}),\] and the statement follows.
\end{proof}

\begin{rem}
The condition in Proposition \ref{proposition characterization of Shalika model in terms of Green} is also equivalent to 
$\overline{\pi}_{\overline{\chi}}\simeq \alpha \otimes \overline{\pi}_{\overline{\chi}}^\vee$. 
\end{rem}

\subsection{Double cosets contributing to distinction}\label{section distinction}

Take $\Delta\in F^\times$ with square root $\d$ generating $E/F$, which we take of valuation $0$ when $E/F$ is unramified and of valuation $1$ when $E/F$ is ramified. The subgroup $H$ of $\GL_n(F)$ consists of invertible matrices of the form $\begin{pmatrix} a & b \\ \D b & a\end{pmatrix}.$  The character $\mu$ of $H$ satisfies
\[\mu\begin{pmatrix} a &  b \\ \D b & a\end{pmatrix}=\mu(\det(a+\d b)).\]
First we identify a non trivial double coset contributing to distinction when $\mu$ has conductor $\geq 2$. Note that when $E/F$ is ramified, if $\mu$ has conductor $l\geq 2$ and is trivial on $1+\P_F$, then 
it has an even conductor, because of the isomorphism $\overline{x}\mapsto 1+\w_F^d\overline{x}$ between 
$k_F=k_E$ and $\frac{1+\P_F^{2d}}{1+\P_F^{2d+1}}$ for any $d\geq 1$.

\begin{prop}\label{propositon example of non trivial double coset contribution}
Suppose $\mu$ has conductor $r+1\geq 2$ but satisfies $\mu(1+\P_F)=1$. We set $l=r$ if $E/F$ is unramified, whereas we set 
$l=(r-1)/2$ when $E/F$ is ramified. Set $d_l=\diag(\w_F^{l}I_m,I_m)$ and suppose that $\chi_{|L_0^*}=\mu_{|F^*} \circ N_{L_0/F}$, then \[\Hom_{\K\cap d_l^{-1}Hd_l}(\l_{\chi},\mu^{d_l})\neq 0,\] where $\mu^{d_l}(x)=\mu(d_l x d_l^{-1})$.
\end{prop}
\begin{proof}
First the condition $\chi_{|L_0^*}=\mu_{|F^*} \circ N_{L_0/F}$ implies that $\chi_{|F^*}=\mu_{|F^*}^m$, hence  
\[\Hom_{\K\cap d_l^{-1}H d_l}(\l_{\chi},\mu^{d_l})\simeq \Hom_{K\cap d_l^{-1}H d_l}(\l_{\chi},\mu^{d_l}).\]
The group $K\cap d_l^{-1}H d_l$ is the set of matrices \[\begin{pmatrix} a & \w_F^{-l}  b \\ \w_F^{l} \D b & a\end{pmatrix}\] with 
$a\in \GL_m(\O_F)$ and $b\in \M_m(\P_F^l)$, and 
\[\mu^{d_l}\begin{pmatrix} a & \w_F^{-l} b \\ \w_F^{l} \D b & a\end{pmatrix}=\mu(\det(a + \d b)).\] But 
\[\det(a +\d b)=\det(a)\det(I_m +\d a^{-1}b))= \det(a)(1+\Tr(\d a^{-1}b))[\M_m(\P_F^{l+1})]\] so 
\[\mu^{d_l}\begin{pmatrix} a & \w_F^{-l} b \\ \w_F^{l} \D b & a\end{pmatrix}=\mu(\det(a))\mu(1+\Tr(\d a^{-1}b)),\] where the dependences are in fact in $\overline{a}\in \GL_m(k_F)$ and $\overline{b}\in \M_m(\P_F^l/\P_F^{l+1})$. So in fact for $a\in \GL_m(\O_F)$ and $b\in \M_m(\O_F)$ we have 
 \[\mu^{d_l}\begin{pmatrix} a & b \\ \w_F^{2l} \D b & a\end{pmatrix}=\overline{\mu^{d_l}}\begin{pmatrix} \overline{a} & \overline{b} \\  & \overline{a}\end{pmatrix}=\overline{\mu_{|F^*}}(\det(a))\mu(1+\w_F^l \d \Tr(\overline{a}^{-1}\overline{b})).\]
 The character $\psi(x)=\mu(1+\w_F^l\d  \overline{x})$ is a nontrivial character of $k_F$ because $\mu(1+\P_F)=1$ 
 whereas $\mu$ has conductor $r+1$. On the other hand \[\l_\chi\begin{pmatrix} a & b \\  \w_F^{2l}\D b & a\end{pmatrix}=\overline{\pi}_{\overline{\chi}}\begin{pmatrix} \overline{a} & \overline{b} \\  & \overline{a}\end{pmatrix}.\]
 Hence $\overline{\pi}_{\overline{\chi}}$ has a $\alpha$-twisted Shalika model and the result follows from Proposition \ref{proposition characterization of Shalika model in terms of Green}.
\end{proof}

\subsection{Multiplicity one when $E/F$ is unramified}

In this section $E/F$ is ramified. We denote by $\Lambda_m^+$ the sequences of integers $(\l_1,\dots ,\l_m)$ with $\l_1\geq \dots \geq \l_m\geq 0$ in $\mathbb{Z}^m$, and set for 
$\l\in \Lambda_m^+$:
\[d_{\l}=\diag(\w_F^{\l_1},\dots,\w_F^{\l_m},1,\dots,1)\in G.\] We recall from \cite{O.04} the following result:

\begin{prop}\label{proposition double coset decomposition}
\[G=\coprod_{\l\in \Lambda_m^+} K d_{\l} H.\]
\end{prop}
\begin{proof}
For $\l\in \Lambda_m^+$ we set $\w_F^\l=\diag(\w_F^{\l_1},\dots,\w_F^{\l_m})$, we also set 
\[w_m=\begin{pmatrix} & & 1 \\ &\udots & \\ 1 & &\end{pmatrix}\in \GL_m(F)\] and 
$w=\diag(I_m,w_m)$. It follows from \cite{G.97} that the map $p:x\mapsto xAx^{-1}$ identifies $G/H$ with the conjugacy class of 
$A$. The matrix 
$d_{\l}$ is sent by $p$ to $\begin{pmatrix} & \w_F^{\l}\\ \Delta \w_F^{-\l} & \end{pmatrix}$, the result now follows from 
\cite[Proposition 4]{O.04}, noting that the group $H$ in \cite{O.04} is equal the centralizer of $wAw^{-1}$ whereas here it is the centralizer of $A$. 
\end{proof}

One has the following multiplicity one result:

\begin{prop}\label{proposition multiplicity one unramified}
Let $\pi(\chi)$ be a cuspidal representation of $\GL_n(F)$ of depth zero for $n\geq 4$, and suppose that $E/F$ is unramified. If it is $\mu$-distinguished, then 
$\Hom_{\GL_m(E)}(\pi(\chi),\mu)\simeq \C$.
\end{prop}
\begin{proof}
Suppose that $\pi(\chi)$ is $\mu$-distinguished so that $\mu_{|F^*}=\alpha\circ N_{L_0/F}$ thanks to Corollary \ref{corollary distinguished determines restriction to L0}. The result follows from Theorem \ref{theorem distinction when mu is tame} when $\mu$ is 
tame so we suppose that $\mu$ has conductor $l+1\geq 2$. By Mackey theory, the result will follow from Propositions \ref{proposition characterization of Shalika model in terms of Green}, \ref{propositon example of non trivial double coset contribution} 
and \ref{proposition double coset decomposition} if we show that if 
\[\Hom_{\K\cap d_{\l}^{-1}Hd_{\l}}(\l_{\chi},\mu^{d_{\l}})=\Hom_{K\cap d_{\l}^{-1}Hd_{\l}}(\l_{\chi},\mu^{d_{\l}})\neq 0\] for $\l\in \Lambda_m^+$, then 
$\l=(l,\dots,l)$. Note that $K\cap d_l^{-1}H d_l$ is the set of matrices \[\begin{pmatrix} a & \w_F^{-\l}  b \\ \w_F^{\l} \D b & a\end{pmatrix}\] with 
$a\in \GL_m(\O_F)$ and $l_i(b)\in (\P_F^{\l_i})^m$ for $i=1,\dots,m$, where $l_i(b)$ is $i$-th row of $b$. So we assume that 
$\Hom_{K\cap d_{\l}^{-1}Hd_{\l}}(\l_{\chi},\mu^{d_{\l}})\neq 0$. 

Suppose first that $\l_m\leq l-1$ and denote by 
$\M_m(\O_F)^-$ the space of matrices in $\M_m(\O_F)$ with $l_i(b)=0$ for $i=1,\dots,m-1$ and $l_m(b)\in (\P_F^{\l_m+1})^m$. Because $\pi(\chi)$ is tame, if $\Hom_{K\cap d_{\l}^{-1}Hd_{\l}}(\l_{\chi},\mu^{d_{\l}})$ was nonzero this would imply that 
\[1=\mu^{d_l}\begin{pmatrix} I_m & \w_F^{-\l} b \\ \w_F^{\l} \D b & I_m\end{pmatrix}=\mu(\det(I_m + \d l_m(b)))\] for 
all $b\in \M_m(\O_F)^-$, hence that $\mu(1+\d \P_F^l)=\{1\}$. Because $\mu(1+\P_F)=\{1\}$ as well, this would in turn imply that 
$\mu(1+\P_E^l)=\mu(1+\P_F^l+\d\P_F^l)=\{1\}$, contradicting the definition of $l$, hence $\l_m\geq l$. 
Now let $s$ be the smallest integer between $1$ and $m$ such that $\l_s=\l_m$, by the arguments of Proposition \ref{propositon example of non trivial double coset contribution} we obtain that 
\[\mu^{d_\l}\begin{pmatrix} a & b \\ \w_F^{2\l} \D b & a\end{pmatrix}=\mu(\det(a))\mu(1+\Tr(\d a^{-1}\w_F^{\l}b))\] for 
$a\in\GL_m(\O_F)$ and $b\in \M_m(\O_F)$. By reduction we deduce that 
\[\mu^{d_\l}\begin{pmatrix} \overline{a} & \overline{b} \\  & \overline{a}\end{pmatrix}=\mu(\det(\overline{a}))\mu(1+\Tr(\d \overline{a}^{-1}\w_F^{\l}\overline{b}))=\mu(\det(\overline{a}))\mu(1+\w_F^{\l_m}\Tr(\d \overline{a}^{-1}\diag(0_{s-1},I_{m-s+1})\overline{b}))\] for $a\in \GL_m(k_F)$ and $b\in \M_m(k_F)$. However the identity 
\[\overline{\pi}_{\overline{\chi}}\begin{pmatrix} \overline{a} & \overline{b} \\  & \overline{a}\end{pmatrix}= \mu^{d_\l}\begin{pmatrix} \overline{a} & \overline{b} \\  & \overline{a}\end{pmatrix}\mathrm{Id}\] first implies that if $\l_m>l$ then 
the unipotent radical of type $(m,m)$ acts trivially on the space $\overline{\pi}_{\overline{\chi}}$ contradicting its cuspidality, hence $\l_m=l$. It also implies that 
\[\overline{b}\mapsto \mu(1+\d \w_F^{\l_m}\Tr(\diag(0_{s-1},I_{m-s+1})\overline{b}))\] must be invariant under conjugation by 
$\GL_m(k_F)$, which in turn implies that $s=1$ hence $\l_1=\dots=\l_m=l$
\end{proof}

\begin{rem}
A similar analysis could certainly be done when $E/F$ is ramified but we don't have at our disposal the description of the double coset representatives given by \cite{O.04} in the unramified case. As we can still prove the Prasad and Takloo-Bighash conjecture in this case, without computing the exact multiplicity, we do not pursue this direction.
\end{rem}

\subsection{Characterization of distinction of level zero cuspidal representations}

The spaces $\Hom_{\K\cap d_{l'}^{-1}Hd_{l'}}(\l_{\chi},\mu^{d_{l'}})$ are isomorphic to a subspace of $\Hom_{H}(\pi(\chi),\mu)$ thanks to Mackey theory for compact induction from open subgroups. Hence as a corollary of Theorem \ref{theorem distinction when mu is tame}, Corollary \ref{corollary distinguished determines restriction to L0}, Propositions \ref{propositon example of non trivial double coset contribution} and \ref{proposition multiplicity one unramified}, we deduce the all assertions of the following theorem.

\begin{thm}\label{theorem characterization of distinction}
For $n\geq 4$, the depth-zero cuspidal representation $\pi(\chi)$ of $\GL_n(F)$ is $\mu$-distinguished if and only if $\chi_{|L_0^*}=\mu_{|F^*} \circ N_{L_0/F}$, except when $E/F$ is ramified and $\mu$ is tame, in which case $\pi(\chi)$ is never $\mu$-distinguished. When $\mu$ is tame or $E/F$ is unramified, the dimension of $\Hom_H(\pi(\chi),\mu)$ is one when nonzero.
\end{thm}

\section{On $\mu_{|F^*}$-symplecticity for Langlands parameters}\label{section mu simplecticity}

In this section, we will freely identify characters of the Weil group $W_K$ and characters of $K^*$ for 
any finite extension $K$ of $F$ via the Artin map. 
We fix $\alpha$ a character of $F^*$. For $\phi$ a finite dimensional irreducible representation of $W_F$, we say that 
$\phi$ is $\alpha$-\textit{selfdual} if \[\phi\simeq \alpha\otimes \phi^\vee.\] 

This is equivalent to say that there exists a non-zero bilinear form $B$ (necessarily non-degenerate) which satisfies 
\[B(\phi(w)v,\phi(w)v')=\alpha(w)B(v,v')\] for all $v, \ v'$ in $V_{\phi}$ and $w\in W_F$. By Schur's Lemma the space of such bilinear forms $B$ is one dimensional hence $B$ is either symmetric or alternate, but not both. In the first case we say that 
$\phi$ is $\alpha$-\textit{orthogonal} and in the second case we say that it is $\alpha$-\textit{symplectic}. As as we shall see later, one way to discriminate between $\alpha$-orthogonal and $\alpha$-symplectic 
tame irreducible representations is the determinant. The proof of the following lemma which we present here is that of the referee, strictly speaking we only need its $m=1$ case to obtain Theorem \ref{theorem dichotomy for tame reps} which is the main result of the section.

\begin{lemma}\label{lemma determinant}
Let $\phi$ be an irreducible representation of $W_F$ of dimension $n$ which is $\alpha$-symplectic, then $\det(\phi)=\alpha^m.$ Moreover 
if $m=1$ the converse is true. 
\end{lemma}
\begin{proof}
Let $B$ be the nonzero alternate form on the space $V$ of $\phi$ with respect to which $\phi$ is $\alpha$-symplectic. View $B$ as an element of the 
exterior product $\bigwedge^2 V^*$ (where $V^*$is the dual of $V$), because $B$ is non degenerate the vector $\bigwedge^m B$ is a nonzero element of the line $\bigwedge^n V^*$ on which $W_F$ acts at the same time by $\det(\phi)$ and $\alpha^m$. The converse when $m=1$ is also clear now. 
\end{proof}
 
We recall that irreducible representations of the form $\Ind_{W_L}^{W_F}(\chi')$ with $\chi'(1+\P_L)=1$ are exactly the \textit{tame} $n$-dimensional irreducible representation of $W_F$, i.e. those with trivial restriction to the wild inertia subgroup of $W_F$. Characterizing $\alpha$-selfdual tame irreducible representations of $W_F$ is easy for $n\geq 4$.

\begin{lemma}\label{lemma selfdual tame reps}
Let $\phi=Ind_{W_L}^{W_F}(\chi')$ be a tame irreducible representation of $W_F$ of dimension $n\geq 4$. Then $\phi$ is $\alpha$-selfdual if and only if $\chi'\circ N_{L/L_0}=\alpha\circ N_{L/F}$, i.e. if and only if 
$\chi'_{|L_0^*}=\alpha\circ N_{L_0/F}$ or $\chi'_{|L_0^*}=\alpha\circ N_{L_0/F}.\eta_{L/L_0}$ where $\eta_{L/L_0}$ is the unramified quadratic character of $L_0^*$. 
\end{lemma}
\begin{proof}
One direction is obvious. For the other if $\phi$ is $\alpha$-selfdual, then $\chi'^{\gamma}\simeq \alpha \circ N_{L/F} \chi'^{-1}$ for some $\gamma\in \Gal_F(L)$. We conclude as in the proof of Corollary \ref{corollary distinguished determines restriction to L0} that $\gamma=\sigma_{L/L_0}$ the Galois involution attached to $L/L_0$. 
\end{proof}

A less obvious task is to distinguished between $\alpha$-symplectic and $\alpha$-orthogonal representations in the statement above. The following lemma, which was given to us by the referee, turns out to be useful for this.

\begin{lemma}\label{lemma ref}
Let $K$ be a subgroup of finite index of a group $G$, $\alpha: G\rightarrow \C^*$ a character. Let $W$ be a finite dimensional representation of $K$ and 
$V=Ind_K^G(W)$. If $W$ is $\alpha_{|K}$-selfdual, resp. $\alpha_{|K}$-orthogonal, resp. $\alpha_{|K}$-symplectic, then $V$ is $\alpha$-selfdual, resp. $\alpha$-orthogonal, resp. $\alpha$-symplectic.
\end{lemma}
\begin{proof}
Take $B$ a non-degenerate bilinear form on $W$ such that $B(k.w,k.w')=\alpha(k)B(w,w')$ for $(k,w,w')\in K\times W \times W$. Then the bilinear form defined on $V\times V$ by 
\[\widetilde{B}(f,f')=\sum_{g\in K\backslash G} \alpha^{-1}(g)B(f(g),f'(g))\] is non-degenerate (because for each $g_0\in G$ and $w\in W$ there is a function $f_{g_0,w} \in Ind_K^G(W)$ supported on $Kg_0$ such that $f_{g_0,w}(g_0)=w$), $\alpha$-equivariant on $V\times V$, of the same type as $B$.
\end{proof}

Lemmas \ref{lemma determinant}, \ref{lemma selfdual tame reps} and \ref{lemma ref} imply the following theorem (when $\alpha$ is trivial similar arguments are used in \cite[Section 6.1]{BHS.17} where Lemma \ref{lemma ref} is tacitly used). We thank the referee for the simplifications 
provided in its proof. 

\begin{thm}\label{theorem dichotomy for tame reps}
Let $\phi=Ind_{W_L}^{W_F}(\chi')$ be a tame irreducible representation of $W_F$ of dimension $n\geq 4$. Then $\phi$ is $\alpha$-symplectic if and only if $\chi'_{|L_0^*}=\alpha\circ N_{L_0/F}.\eta_{L/L_0}$.
\end{thm}
\begin{proof}
Let's first have a look at $\phi_0:=Ind_{W_L}^{W_{L_0}}(\chi'_{|L_0^*})$. According to Lemma \ref{lemma determinant} the representation $\phi_0$ is $\alpha\circ N_{L_0/F}$-symplectic if and only if its determinant is equal to $\alpha\circ N_{L_0/F}$. This is well-known to be the same as $\chi'_{|L_0^*}=\alpha\circ N_{L_0/F}.\eta_{L/L_0}$ (see for example \cite[29.2 Proposition]{BH.06}). In particular if $\chi'_{|L_0^*}=\alpha\circ N_{L_0/F}.\eta_{L/L_0}$ then $\phi$ is $\alpha$-symplectic thanks to Lemma \ref{lemma ref}. 
Conversely if $\phi$ is $\alpha$-symplectic then by Lemma \ref{lemma selfdual tame reps} we have $\chi'_{|L_0^*}=\alpha\circ N_{L_0/F}$ or $\chi'_{|L_0^*}=\alpha\circ N_{L_0/F}.\eta_{L/L_0}$. If we were in the first case then $\phi_0$ would be $\alpha\circ N_{L_0/F}$-orthogonal hence $\phi$ in turn $\alpha$-orthogonal by Lemma \ref{lemma ref}, which is not possible as $\phi$ is irreducible. Hence $\chi'_{|L_0^*}=\alpha\circ N_{L_0/F}.\eta_{L/L_0}$.
\end{proof}

\begin{rem}
In fact Lemma \ref{lemma determinant} allows another criterion to discriminate between $\alpha$-symplectic and $\alpha$-orthogonal representations in Lemma \ref{lemma selfdual tame reps}, namely $\phi$ is $\alpha$-symplectic if and only if $\det(\phi)=\alpha^m$. Indeed if we had $\chi'_{|L_0^*}=\alpha\circ N_{L_0/F}$, retricting to $F^*$ we would get $\chi'_{|F^*}=\alpha^m$ and this would contradict Lemma \ref{lemma determinant} because according to \cite[29.2 Proposition]{BH.06} one has $det(\phi)=\eta_{K/F}\chi'_{|F^*}$ for $K/F$ the quadratic unramified extension of $F$ and $\eta_{K/F}$ its corresponding quadratic character.
\end{rem}

The Langlands parameter of the representation $\pi(\chi)$is given by \cite[Theorem 2]{BH.11}: 
it is \[\phi(\pi(\chi)):=\Ind_{W_L}^{W_F}(\eta\chi)\] where $\eta$ is the unramified quadratic character of $L^*$. Hence Theorem \ref{theorem dichotomy for tame reps} has the following immediate corollary.

\begin{cor}\label{corollary mu-symplecticity}
For $n\geq 4$ (i.e. $m\geq 2$), the representation $\pi(\chi)$ is $\mu_{|F^*}$-symplectic if and only if $\chi_{|L_0^*}=\mu_{|F^*} \circ N_{L_0/F}$.
\end{cor}

\section{The Prasad and Takloo-Bighash conjecture}\label{section PTB}

We recall that the conjecture of Prasad and Takloo-Bighash has been proved by Tunnel and also Saito when $n=2$ (\cite[Theorem p.1277]{T.83} in residual characteristic not $2$, \cite[Theorem p.99]{S.93} in characteristic not $2$), \textit{hence in this Section we assume $n\geq 4$}. So comparing the statements of Theorem \ref{theorem characterization of distinction} and Corollary \ref{corollary mu-symplecticity}, it is enough to compute the Prasad and Takloo-Bighash $\epsilon$ value of a cuspidal depth-zero representation $\pi(\chi)$ with $\chi_{|L_0^*}=\mu_{|F^*} \circ N_{L_0/F}$, and to 
show that it is as expected by the conjecture when $E/F$ is unramified or $E/F$ is ramified and $\mu$ is not tame, and differs from the expected value when $E/F$ is ramified and $\mu$ is tame. 
Again in the proof we will freely confuse characters of Weil groups and of multiplicative groups of local fields (hence restrictions will be often written as composition with the norm map).\\

\noindent Let's do some preliminary computations before computing the $\epsilon$ factor of the Prasad and Takloo-Bighash conjecture. 
When $E/F$ is \textit{unramified} we have:\vspace{-1em}

\begin{eqnarray*}
& & \Ind_{W_L}^{W_F}(\eta\chi)\otimes \Ind_{W_E}^{W_F}(\mu^{-1})\\
 &=&\Ind_{W_L}^{W_F}(\eta\chi \otimes {\Ind_{W_E}^{W_F}(\mu^{-1})}_{|W_L})\\
&=&\Ind_{W_L}^{W_F}(\eta\chi(\mu^{-1} \circ N_{L/E}) \oplus \eta\chi(\mu^{-\sigma_{E/F}} \circ N_{L/E}))\\
&&\text{by Mackey's restriction formula with }<\sigma_{E/F}>=\mathrm{Gal}_F(E)\\
&=&\Ind_{W_L}^{W_F}(\eta\chi(\mu^{-1} \circ N_{L/E})) \oplus \Ind_{W_L}^{W_F}(\eta\chi(\mu^{-\sigma_{E/F}} \circ N_{L/E})).
\end{eqnarray*}

When $E/F$ is \textit{ramified} we have:\vspace{-1em} 

\begin{eqnarray*}
&&\Ind_{W_L}^{W_F}(\eta\chi)\otimes \Ind_{W_E}^{W_F}(\mu^{-1})\\
 &=&\Ind_{W_L}^{W_F}(\eta\chi\otimes \Ind_{W_E}^{W_F}(\mu^{-1})_{|W_L})\\
 &=&\Ind_{W_E}^{W_F}(\eta\chi\otimes \Ind_{W_M}^{W_L}(\mu^{-1}\circ N_{M/E}))\\
 &&\text{by Mackey's restriction formula with }M=<L,E>\\
 &=&\Ind_{W_M}^{W_F}((\eta\chi)\circ N_{M/L}.\mu^{-1}\circ N_{M/E}).
\end{eqnarray*}

\begin{thm}\label{theorem PTB epsilon value}
 Let $\pi(\chi)$ be a depth-zero cuspidal representation of $\GL_n(F)$, such that $\chi_{|L_0^*}=\mu_{|F^*} \circ N_{L_0/F}$. Let $\psi$ be a non-trivial additive character of $F$.
 \begin{itemize}[label=\textbullet]
  \item If $E/F$ is unramified, then $\epsilon(\frac{1}{2},\pi(\chi)\otimes \Ind_{W_E}^{W_F}(1),\psi)=\omega_{E/F}(-1)^m\mu(-1)^m$. 
  \item If $E/F$ is ramified:
  \begin{itemize}
   \item If $\mu$ is tame then $\epsilon(\frac{1}{2},\pi(\chi)\otimes \Ind_{W_E}^{W_F}(1),\psi)=-\omega_{E/F}(-1)^m\mu(-1)^m$. 
   \item If $\mu$ is not tame then $\epsilon(\frac{1}{2},\pi(\chi)\otimes \Ind_{W_E}^{W_F}(1),\psi)=\omega_{E/F}(-1)^m\mu(-1)^m$. 
 \end{itemize}
 \end{itemize}
\end{thm}

\begin{proof}

If $L/K$ is a separable quadratic extension of non Archimedean local fields, we denote by $\sigma_{L/K}$ the associated Galois involution. We distinguish the ramified and the unramified case in our computations. 

\textbf{When $E/F$ is unramified.} We recall the situation: $E$ is included in $L$ and possibly in $L_0$ according to the 
  parity of $m$.
\vspace{-0.5cm}

\begin{figure}[H]
\centering
\begin{minipage}[t]{6cm}
\centering
\includegraphics[width=5.5cm]{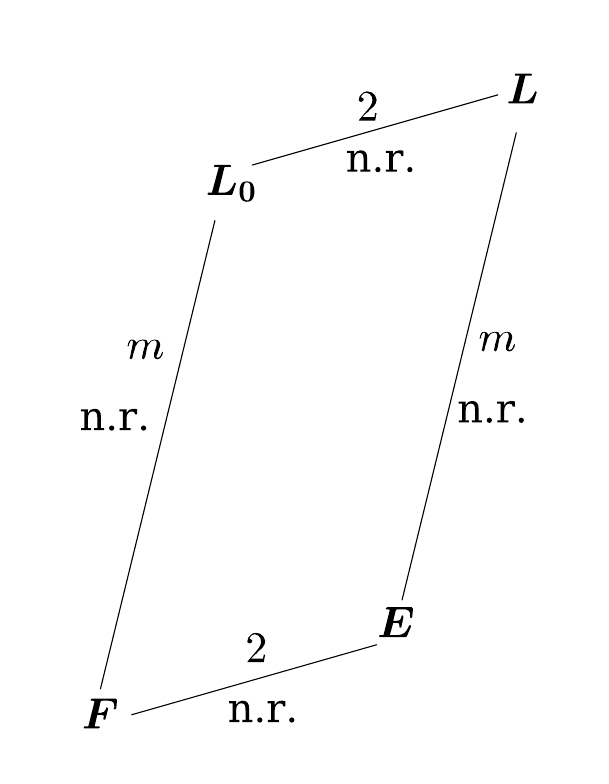}
\end{minipage}
\begin{minipage}[t]{6cm}
\centering
\includegraphics[width=4cm]{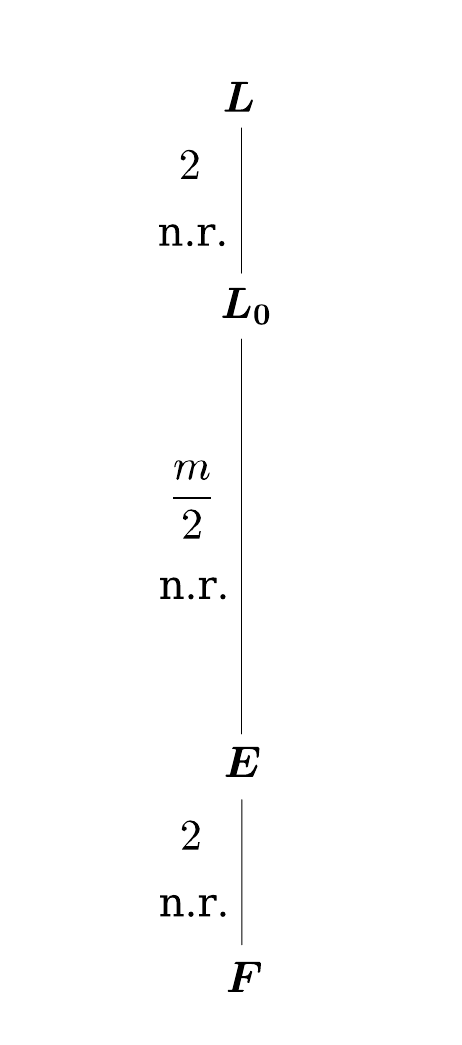}
\end{minipage}
\captionsetup{justification=centering}
\caption[]{Diagram of the extensions involved - $E/F$ unramified case \\ $m$ odd (in the left) and $m$ even (in the right)}
\end{figure}

\begin{eqnarray*}
& &\epsilon(\frac{1}{2},\Ind_{W_L}^{W_F}(\eta\chi)\otimes \Ind_{W_E}^{W_F}(\mu^{-1}),\psi)\\
&=&\epsilon(\frac{1}{2},\Ind_{W_L}^{W_F}(\eta\chi(\mu^{-1}\circ N_{L/E})),\psi)\epsilon(\Ind_{W_L}^{W_F}(\eta\chi(\mu^{-\sigma_{E/F}}\circ N_{L/E})),\psi)\text{ by §\ref{section properties of constants}, \ref{equation additivity of epsilon}.}\\
 &=&\lambda_{L/F}^{2}(\psi)\epsilon(\frac{1}{2},\eta\chi(\mu^{-1}\circ N_{L/E}),\psi_L)\epsilon(\frac{1}{2},\eta\chi(\mu^{-\sigma_{E/F}}\circ N_{L/E}),\psi_L)\text{ by §\ref{section properties of constants}, \ref{equation inductivity}.}\\
&=&\lambda_{L/E}^{2}(\psi_E)\lambda_{E/F}^{n}(\psi)\eta^2(\w_L^{d(\psi_L)})\epsilon(\frac{1}{2},\chi(\mu^{-1}\circ N_{L/E}),\psi_L)\epsilon(\frac{1}{2},\chi(\mu^{-\sigma_{E/F}}\circ N_{L/E}),\psi_L)\\
 & &\text{ by §\ref{section properties of constants}, \ref{equation multiplicativité de la constante de Langlands}.}\\
&=&\omega_{E/F}(-1)^{m}\epsilon(\frac{1}{2},\chi(\mu^{-1}\circ N_{L/E}),\psi_L)\epsilon(\frac{1}{2},\chi(\mu^{-\sigma_{E/F}}\circ N_{L/E}),\psi_L)\\
& &\text{ by §\ref{section properties of constants}, \ref{equation square=-1}. and \ref{equation constante de Langlands NR}. and because $n$ is even.}
\end{eqnarray*} 

Now we distinguish between two cases:

\begin{enumerate}
\item $m$ is even: then 
\begin{eqnarray*}
& &\epsilon(\frac{1}{2},\chi(\mu^{-1}\circ N_{L/E}),\psi_L)\epsilon(\frac{1}{2},\chi(\mu^{-\sigma_{E/F}}\circ N_{L/E}),\psi_L)\\
&=&\epsilon(\frac{1}{2},\chi^{\sigma_{L/L_0}}(\mu^{-1}\circ N_{L/E}),\psi_L)\epsilon(\frac{1}{2},\chi(\mu^{-\sigma_{E/F}}\circ N_{L/E}),\psi_L)\\
& &\text{according to §\ref{section properties of constants}, \ref{equation galois invariance of epsilon}. because $\psi_L=\psi_L^{\sigma_{L/L_0}}$}\\
& &\text{and $\mu^{-1}\circ N_{L/E}$ is also $\sigma_{L/L_0}$-invariant as $E\subset L_0\subset L$.}\\
&=&\epsilon(\frac{1}{2},\chi^{\sigma_{L/L_0}}(\mu^{-1}\circ N_{L/E}),\psi_L^{-1})\epsilon(\frac{1}{2},\chi(\mu^{-\sigma_{E/F}}\circ N_{L/E}),\psi_L)\\
& &\text{from §\ref{section properties of constants}, \ref{equation translation caractère additif} because} (\chi^{\sigma_{L/L_0}}(\mu^{-1}\circ N_{L/E}))(-1)=\\
& & (\mu_{|F^*} \circ N_{L_0/F})(-1)(\mu^{-1}\circ N_{L/E})(-1)=\mu(-1)^{m}\mu(-1)^{-m}=1
\end{eqnarray*} 
But then because \[\chi^{\sigma_{L/L_0}}(\mu^{-1}\circ N_{L/E})\chi(\mu^{-\sigma_{E/F}}\circ N_{L/E})=\chi\circ N_{L/L_0}.\mu_{|F^*}^{-1}\circ N_{L/F}=\mu_{|F^*} \circ N_{L/F}.\mu_{|F^*}^{-1}\circ N_{L/F}=1,\] 
§\ref{section properties of constants}, \ref{equation epsilon times epsilon dual}. implies that 
\[\epsilon(\frac{1}{2},\chi^{\sigma_{L/L_0}}(\mu^{-1}\circ N_{L/E}),\psi_L^{-1})\epsilon(\frac{1}{2},\chi(\mu^{-\sigma_{E/F}}\circ N_{L/E}),\psi_L)=1,\]
and we recognize the expected value $\epsilon(\frac{1}{2},\Ind_{W_L}^{W_F}(\eta\chi)\otimes \Ind_{W_E}^{W_F}(\mu^{-1}),\psi)=\omega_{E/F}^m(-1)^m\mu(-1)^m$ because $m$ is even.
\item $m$ is odd: then we notice that both 
$\chi(\mu^{-1}\circ N_{L/E})$ and $\chi(\mu^{-\sigma_{E/F}}\circ N_{L/E})$ restrict to $L_0^*$ as 
$\chi_{|L_0^*}(\mu_{|F^*} \circ N_{L_0/F})=1$. Hence by §\ref{section properties of constants}, \ref{equation FQ}, for $v\in L-L_0$ such that $v^2\in L_0$, we have 
\[\epsilon(\frac{1}{2},\chi(\mu^{-1}\circ N_{L/E}),\psi_L)=\chi(v)\mu^{-1}(N_{L/E}(v))\]
and 
\[\epsilon(\frac{1}{2},\chi(\mu^{-\sigma_{E/F}}\circ N_{L/E}),\psi_L)=\chi(v)\mu^{-\sigma_{E/F}}(N_{L/E}(v)),\] 
so that 
\[\epsilon(\frac{1}{2},\chi(\mu^{-1}\circ N_{L/E}),\psi_L)\epsilon(\frac{1}{2},\chi(\mu^{-\sigma_{E/F}}\circ N_{L/E}),\psi_L)= 
\chi(v^2)\mu^{-1}(N_{L/F}(v))\]
\[=\mu(N_{L_0/F}(v^2)N_{L/F}(v)^{-1})=\mu(N_{L_0/F}(v^2N_{L/L_0}(v)^{-1}))=\mu(N_{L_0/F}(-1))\] because $\sigma_{L/L_0}(v)=-v$, hence finally 
\[\epsilon(\frac{1}{2},\Ind_{W_L}^{W_F}(\eta\chi)\otimes \Ind_{W_E}^{W_F}(\mu^{-1}),\psi)=\omega_{E/F}(-1)^m\mu(-1)^m\] which is again the expected value.
\end{enumerate}

  \textbf{When $E/F$ is ramified.} In this case, $E$ is not included in $L$. Set $M$ to be the extension of $L$ generated by  
  $L$ and $E$, $M$ is therefore unramified $n$-dimensional on $E$. We also set $L_1=<E,L_0>$ so that $M$ is an unramified quadratic extension of $L_1$. The situation is as follows.\\
 \begin{figure}[h]
 \begin{center}
 \includegraphics[width=0.4\textwidth]{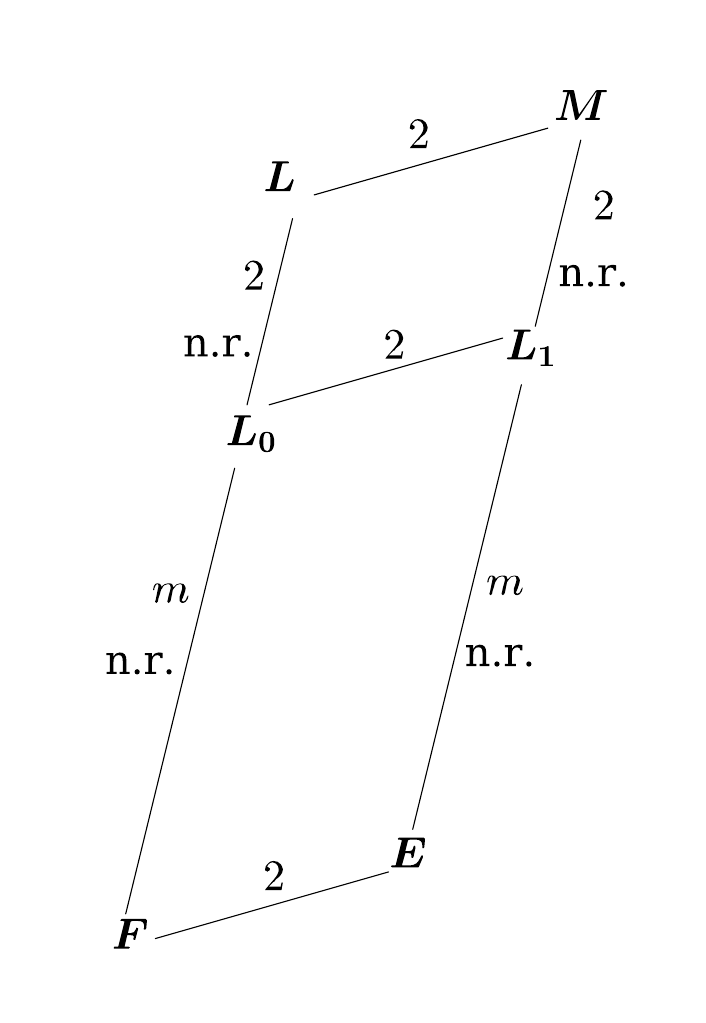}
 \caption[le titre]{Diagram of the extensions involved - $E/F$ ramified case}
 \end{center}
 \end{figure}
 \vspace{-0,2cm}

\begin{eqnarray*}
 & &\epsilon(\frac{1}{2},\Ind_{W_L}^{W_F}(\eta\chi)\otimes \Ind_{W_E}^{W_F}(\mu^{-1}),\psi)\\ 
 &=&\epsilon(\frac{1}{2},\Ind_{W_M}^{W_F}((\eta\chi)\circ N_{M/L}.\mu^{-1}\circ N_{M/E}),\psi)\\
 &=&\lambda_{M/F}(\psi)\epsilon(\frac{1}{2},(\eta\chi)\circ N_{M/L}.\mu^{-1}\circ N_{M/E},\psi_M) \text{ by §\ref{section properties of constants}, \ref{equation inductivity}}.\\
 &=&\lambda_{M/E}(\psi_E)\lambda_{E/F}^n(\psi)\epsilon(\frac{1}{2},(\eta\chi)\circ N_{M/L}.\mu^{-1}\circ N_{M/E},\psi_M) 
 \text{ by §\ref{section properties of constants}, \ref{equation multiplicativité de la constante de Langlands}.}\\
 &=&(-1)^{d(\psi_E)(n-1)}\omega_{E/F}(-1)^m\epsilon(\frac{1}{2},\omega_{M'/M}.\chi\circ N_{M/L}.\mu^{-1}\circ N_{M/E},\psi_M)\\                            & &\text{ by §\ref{section properties of constants}, \ref{equation constante de Langlands NR}. and \ref{equation square=-1}. where $M'/M$ is quadratic unramified.}\\
\end{eqnarray*}

Before proceeding further with the computation let's discuss the conductor of the character $\chi\circ N_{M/L}.\mu^{-1}\circ N_{M/E}$. 
\begin{itemize}
\item If $\mu$ is not tame then $\chi\circ N_{M/L}.\mu^{-1}\circ N_{M/E}$ clearly has the same conductor as $\mu^{-1}\circ N_{M/E}$ which is also not tame as it has the same conductor as $\mu$, by surjectivity of 
$N_{M/E}$ from $1+\P_M^d$ onto $1+\P_E^d$ for any $d\geq 1$. In partcular $\chi\circ N_{M/L}.\mu^{-1}\circ N_{M/E}$ has conductor 
$c(\mu)$ which is even as we saw in Section \ref{section distinction}.
\item If $\mu$ is tame let us show that the character $\chi\circ N_{M/L}.\mu^{-1}\circ N_{M/E}$ has conductor $1$. Clearly it is trivial on $1+\P_M$ because $\chi$ and $\mu$ are tame, but if it was unramified, going backwards one would deduce that $\Ind_{W_L}^{W_F}(\chi)\otimes \Ind_{W_E}^{W_F}(\mu^{-1})$ would be unramified, hence a direct sum of unramified characters. But $\Ind_{W_L}^{W_F}(\chi)\otimes \Ind_{W_E}^{W_F}(\mu^{-1})$ cannot contain any character, otherwise by irreducibility of $\Ind_{W_L}^{W_F}(\chi)$, it would appear as sub-representation of a character twist of $\Ind_{W_E}^{W_F}(\mu)$, 
which is impossible for dimension reasons (remember that we suppose $n\geq 3$). Hence $\chi\circ N_{M/L}.\mu^{-1}\circ N_{M/E}$ has conductor $1$. 
\end{itemize}
Hence setting $c'(\mu)=c(\mu)$ when $c(\mu)\geq 1$ and $c'(\mu)=1$ when $\mu$ is unramified, we obtain $c(\chi\circ N_{M/L}.\mu^{-1}\circ N_{M/E})=c'(\mu)$, which is even as soon as $c'(\mu)>1$. Finally we obtain: 
\begin{eqnarray*}
 & &\epsilon(\frac{1}{2},\Ind_{W_L}^{W_F}(\eta\chi)\otimes \Ind_{W_E}^{W_F}(\mu^{-1}),\psi)\\
 &=&(-1)^{d(\psi_M)(n-1)}\omega_{E/F}(-1)^m\omega_{M'/M}(\varpi_M^{d(\psi_M)+c'(\mu)})\epsilon(\frac{1}{2},\chi\circ N_{M/L}.\mu^{-1}\circ N_{M/E},\psi_M)\\
 & &\text{ thanks to §\ref{section properties of constants}, \ref{equation torsion NR}}\\
&=&(-1)^{d(\psi_M)(n-1)}\omega_{E/F}(-1)^m(-1)^{d(\psi_M)+c'(\mu)}\epsilon(\frac{1}{2},\chi\circ N_{M/L}.\mu^{-1}\circ N_{M/E},\psi_M)\\
&=&(-1)^{c'(\mu)}\omega_{E/F}(-1)^m\epsilon(\frac{1}{2},\chi\circ N_{M/L}.\mu^{-1}\circ N_{M/E},\psi_M)\\
& &\text{because $n$ is even.}
\end{eqnarray*}

Note that $M/L_0$ is bi-quadratic, so there is one more quadratic extension $L_2$ of $L_0$ under $M$. 
Now the restriction of $\chi \circ N_{M/L}$ to $L_2^*$ is equal to $\chi\circ N_{L_2/L_0}=\mu_{|F^*} \circ N_{L_2/F}$, 
 whereas that of 
$\mu^{-1}\circ N_{M/E}$ is equal to $\mu^{-1}\circ N_{L_2/F}$, hence 
$\chi\circ N_{M/L}.\mu^{-1}\circ N_{M/E}$ restricts trivially to $L_2^*$.\\

Take $v\in L\setminus L_0$ with $v^2\in L_0$. Then $M=L_2[v]$ and we can apply §\ref{section properties of constants}, \ref{equation FQ}. : 
\begin{eqnarray*}
&&\epsilon(\frac{1}{2},\chi\circ N_{M/L}.\mu^{-1}\circ N_{M/E},\psi_M)\\
 &=& \chi\circ N_{M/L}(v).\mu^{-1}\circ N_{M/E}(v)\\
 &=& \chi(v^2)\mu^{-1}\circ N_{L_1/E}(-v^2) \\
 &=&\chi(v^2)\mu_{|F^*}^{-1}\circ N_{L_0/F}(-v^2)\\
 &=&\mu_{|F^*} \circ N_{L_0/F}(v^2)\mu_{|F^*}^{-1}\circ N_{L_0/F}(-v^2)\\
 &=&\mu(-1)^m
\end{eqnarray*}

Thus $\epsilon(\frac{1}{2},\Ind_{W_L}^{W_F}(\eta\chi)\otimes \Ind_{W_E}^{W_F}(1),\psi)=(-1)^{c'(\mu)}\omega_{E/F}^m(-1)\mu(-1)^m$, as expected.

\end{proof}

As a corollary, we obtain:

\begin{cor}\label{corollary PTB}
 Let $\pi(\chi)$ be a depth $0$ cuspidal representation of $\GL_n(F)$,
 %(c'est une représentation de la série discrète \textcolor{red}{si et seulement si $\chi$ est unitaire).} 
 let $\mu$ be a character of $E^*$, then $\pi(\chi)$ is $\mu\circ \mathrm{det}_{\GL_m(E)}$-distinguished by 
 $H=\GL_m(E)$ if and only if
 \begin{enumerate}
  \item $\pi(\chi)$ is $\mu_{|F^*}$-symplectic;
  \item $\epsilon(\frac{1}{2},\Ind_{W_L}^{W_F}(\eta\chi)\otimes \Ind_{W_E}^{W_F}(\mu^{-1}))=\omega_{E/F}(-1)^m\mu(-1)^m$.
 \end{enumerate}

\end{cor}

\bibliographystyle{alpha}
\bibliography{ptb}
\end{document}